\documentclass{amsart}
\usepackage{amsmath, amssymb,epic,graphicx,mathrsfs,enumerate}
\usepackage[all]{xy}

\usepackage{amsthm}
\usepackage{amssymb}
\usepackage{latexsym}
\usepackage{longtable}
\usepackage{epsfig}
\usepackage{amsmath}
\usepackage{hhline}

\newtheorem{thm}{Theorem}[section]

\newtheorem{lem}[thm]{Lemma}
\newtheorem{cor}[thm]{Corollary}

\newcommand{\GL}{\operatorname{GL}}

\newcommand{\SL}{\operatorname{SL}}

\begin{document}

\title[Average Dimension]
{Average Dimension of Fixed Point Spaces With Applications}
\author{Robert M. Guralnick}
\address{Department of Mathematics, University of Southern California,
Los Angeles, CA 90089-2532, USA}
\email{guralnic@usc.edu}
\author{Attila Mar\'oti}
\address{MTA Alfr\'ed R\'enyi Institute of Mathematics, Re\'altanoda utca 13-15, H-1053, Budapest, Hungary}
\email{maroti@renyi.hu} 
\thanks{The first author was partially supported by NSF
grant DMS 0653873. The research of the second
author was supported by a Marie Curie International Reintegration
Grant within the 7th European Community Framework Programme and
partially by grants OTKA T049841 and OTKA NK72523.}
\date{\today}

\begin{abstract}
Let $G$ be a finite group, $F$ a field, and $V$ a finite
dimensional $FG$-module such that $G$ has no trivial composition
factor on $V$. Then the arithmetic average dimension of the fixed
point spaces of elements of $G$ on $V$ is at most $(1/p) \dim V$
where $p$ is the smallest prime divisor of the order of $G$. This
answers and generalizes a 1966 conjecture of Neumann which also
appeared in a paper of Neumann and Vaughan-Lee and also as a
problem in The Kourovka Notebook posted by Vaughan-Lee. Our result
also generalizes a recent theorem of Isaacs, Keller,
Meierfrankenfeld, and Moret\'o. Various applications are given.
For example, another conjecture of Neumann and Vaughan-Lee is
proven and some results of Segal and Shalev are improved and/or
generalized concerning BFC groups.
\end{abstract}
\maketitle

\begin{center}
{\it Dedicated to Peter M. Neumann on the occasion of his 70th
birthday.}
\end{center}

\bigskip

\section{Introduction}

Let $G$ be a finite group, $F$ a field, and $V$ a finite
dimensional $FG$-module. For a non-empty subset $S$ of $G$ we
define
$$
\mathrm{avgdim}(S,V) = \frac{1}{|S|} \sum_{s \in S} \dim C_V(s)
$$
to be the arithmetic average dimension of the fixed point spaces
of all elements of $S$ on $V$. In his 1966 DPhil thesis Neumann
\cite{N} conjectured that if $V$ is an irreducible $FG$-module
then $\mathrm{avgdim}(G,V) \leq (1/2) \dim V$. This problem was
restated in 1977 by Neumann and Vaughan-Lee \cite{nv} and was
posted in 1982 by Vaughan-Lee in The Kourovka Notebook
\cite{Kourovka} as Problem 8.5. The conjecture was proved by
Neumann and Vaughan-Lee \cite{nv} for solvable groups $G$ and also
in the case when $|G|$ is invertible in $F$. Later Segal and
Shalev \cite{SS} showed that $\mathrm{avgdim}(G,V) \leq (3/4) \dim
V$ for an arbitrary finite group $G$. This was improved by Isaacs,
Keller, Meierfrankenfeld, and Moret\'o \cite{IKMM} to
$\mathrm{avgdim}(G,V) \leq ((p+1)/2p) \dim V$ where $p$ is the
smallest prime factor of $|G|$. Our first main theorem is

\begin{thm}
\label{t1} Let $G$ be a finite group, $F$ a field, and $V$ a
finite dimensional $FG$-module. Let $N$ be a normal subgroup of
$G$ that has no trivial composition factor on $V$. Then
$\mathrm{avgdim}(Ng,V) \leq (1/p) \dim V$ for every $g \in G$
where $p$ is the smallest prime factor of the order of $G$.
\end{thm}

Theorem \ref{t1} not only solves the above-mentioned conjecture of
Neumann and Vaughan-Lee but it also generalizes and improves the
problem in many ways. First of all, $G$ need not be irreducible on
$V$; the only restriction we impose is that $G$ has no trivial
composition factor on $V$. Secondly, we prove the bound $(1/2)
\dim V$ not just for $\mathrm{avgdim}(G,V)$ but for
$\mathrm{avgdim}(S,V)$ where $S$ is an arbitrary coset of a normal
subgroup of $G$ with a certain property. Thirdly, Theorem \ref{t1}
involves a better general bound, namely $(1/p) \dim V$ where $p$
is the smallest prime divisor of $|G|$. Note that the example
\cite[Page 3129]{IKMM} of a completely reducible $FG$-module $V$
for an elementary abelian $p$-group $G$ shows that
$\mathrm{avgdim}(G,V) = (1/p) \dim V$ can occur in Theorem
\ref{t1}. There are examples for equality in Theorem \ref{t1} even
when $V$ is an irreducible module. Just consider a non-trivial
irreducible representation of $G = C_{p}$ for $p$ an arbitrary
prime. Our other example works for arbitrarily large dimensions.
Let $p$ be an arbitrary odd prime, let $G$ be the extraspecial
$p$-group of order $p^{1+2a}$ (for a positive integer $a$) of
exponent $p$, let $N = Z(G)$, let $F$ be an algebraically closed
field of characteristic different from $p$, and let $V$ be an
irreducible $FG$-module of dimension $p^{a}$. Then for every
element $x \in G \setminus N$ we have $\dim C_{V}(x) = (1/p) \dim
V$ and so $\mathrm{avgdim}(Ng,V) = (1/p) \dim V$ for every $g \in
G$. In particular we have $\mathrm{avgdim}(H,V) = (1/p)\dim V$ for
every subgroup $H$ of $G$ containing $N$.

In his DPhil thesis \cite{N} Neumann showed that if $V$ is a
non-trivial irreducible $FG$-module for a field $F$ and a finite
solvable group $G$ then there exists an element of $G$ with small
fixed point space. Specifically, he showed that there exists $g
\in G$ with $\dim C_{V}(g) \leq (7/18) \dim V$. Neumann
conjectured that in fact, there should exists $g \in G$ such that
$\dim C_{V}(g) \leq (1/3) \dim V$. Segal and Shalev \cite{SS}
proved, for an arbitrary finite group $G$, that there exists an
element $g \in G$ with $\dim C_{V}(g) \leq (1/2) \dim V$. Later,
under milder conditions ($V$ is a completely reducible $FG$-module
with $C_{V}(G) = 0$), Isaacs, Keller, Meierfrankenfeld, and
Moret\'o \cite{IKMM} showed that there exists an element $g \in G$
with $\dim C_{V}(g) \leq (1/p) \dim (V)$ where $p$ is the smallest
prime divisor of $|G|$. Under even weaker conditions we improve
this latter result.

\begin{cor}
\label{c1} Let $G$ be a finite group, $F$ a field, and $V$ a
finite dimensional $FG$-module. Let $N$ be a normal subgroup of
$G$ that has no trivial composition factor on $V$. Let $x$ be an
element of $G$ and let $p$ be the smallest prime factor of the
order of $G$. Then there exists an element $g \in Nx$ with $\dim
C_{V}(g) \leq (1/p) \dim V$ and there exists an element $g \in N$
with $\dim C_{V}(g) < (1/p) \dim V$.
\end{cor}

Note that Corollary \ref{c1} follows directly from Theorem
\ref{t1} just by noticing that $\dim C_{V}(1) = \dim V$. Note also
that if $V$ is irreducible and faithful in Corollary \ref{c1} then
no non-trivial normal subgroup of $G$ has a non-zero fixed point
on $V$ and so the $N$ above can be any non-trivial normal subgroup
of $G$. During the last stage of the writing of this paper
Neumann's above-mentioned conjecture was proved in \cite{gm}; if
$V$ is a non-trivial irreducible $FG$-module for a finite group
$G$ then there exists an element $g \in G$ such that $\dim
C_{V}(g) \leq (1/3) \dim V$.

Let $\mathrm{cl}_{G}(g)$ denote the conjugacy class of an element
$g$ in a finite group $G$, and for a positive integer $n$ and a
prime $p$ let $n_{p}$ denote the $p$-part of $n$. In \cite{IKMM}
Isaacs, Keller, Meierfrankenfeld, and Moret\'o conjecture that for
any primitive complex irreducible character $\chi$ of a finite
group $G$ the degree of $\chi$ divides $|\mathrm{cl}_{G}(g)|$ for
some element $g$ of $G$. Using their result mentioned before the
statement of Corollary \ref{c1} they showed that if $\chi$ is an
arbitrary primitive complex irreducible character of a finite
solvable group $G$ and $p$ is a prime divisor of $|G|$ then
$\chi(1)_{p}$ divides ${(|\mathrm{cl}_{G}(g)|)}^{3}$ for some $g
\in G$. Using Theorem \ref{t1} we may prove more than this.

\begin{cor}
\label{c5} Let $\chi$ be an arbitrary primitive complex
irreducible character of a finite solvable group $G$ and let $p$
be a prime divisor of $|G|$. Then the number of $g \in G$ for
which $\chi(1)_{p}$ divides ${(|\mathrm{cl}_{G}(g)|)}^{3}$ is at
least $(2|G|)/(1+k)$ where $k = \log_{p} |G|_{p}$. Furthermore if
$\chi(1)_{p} > 1$ then there exists a $p'$-element $g \in G$ for
which $p^{3} \cdot \chi(1)_{p}$ divides
${(|\mathrm{cl}_{G}(g)|)}^{3}$.
\end{cor}

Recall that a chief factor of a finite group is a section $X/Y$ of
$G$ with $Y < X$ both normal in $G$ such that there is no normal
subgroup of $G$ strictly between $X$ and $Y$. Note that $X/Y$ is a
direct product of isomorphic simple groups. If $X/Y$ is abelian,
then it is an irreducible $G$-module. If $X/Y$ is non-abelian,
then $G$ permutes the direct factors transitively. A chief factor
is called central if $G$ acts trivially on $X/Y$ and non-central
otherwise. Let $G$ be a finite group acting on another finite
group $Z$ by conjugation. For a non-empty subset $S$ of $G$ define
$$
\mathrm{geom}(S,Z) = {\Bigg( \prod_{s \in S} |C_{Z}(s)|
\Bigg)}^{1/|S|}
$$
to be the geometric mean of the sizes of the centralizers of
elements of $S$ acting on $Z$. Similarly, for a non-empty subset
$S$ of $G$ define
$$
\mathrm{avg}(S,Z) = \frac{1}{|S|} \sum_{s \in S} |C_{Z}(s)|
$$
to be the arithmetic mean of the sizes of the centralizers of
elements of $S$ acting on $Z$. Our next result is a non-abelian
version of Theorem \ref{t1} proved using some recent work of
Fulman and the first author \cite{FG}.

\begin{thm}
\label{t2} Let $G$ be a finite group with $X/Y = M$ a non-abelian
chief factor of $G$ with $X$ and $Y$ normal subgroups in $G$.
Then, for any $g \in G$, we find that $\mathrm{geom}(Xg,M) \leq
\mathrm{avg}(Xg,M) \leq {|M|}^{1/2}$.
\end{thm}

In fact, one can do slightly better than $1/2$ in the exponent of
the statement of Theorem \ref{t2}, but it is a bit easier to write
down the proof of this result. It is easy to see that one cannot
do better than $1/3$ (consider $G=\SL(2,q)$ with $q = 2^e > 2$ --
then all non-trivial elements have centralizers of orders  $q-1$,
$q$, or $q+1$ which are approximately $|G|^{1/3}$).

Let $\mathrm{ccf}(G)$ and $\mathrm{ncf}(G)$ be the product of the
orders of all central and non-central chief factors (respectively)
of a finite group $G$. (In case these are not defined put them
equal to $1$.) These invariants are independent of the choice of
the chief series of $G$.  Let $F(G)$ denote the Fitting subgroup
of $G$.   Note that $F(G)$ acts trivially on every chief factor of $G$.
Using Theorems \ref{t1} and \ref{t2} we
prove

\begin{thm}
\label{t3} Let $G$ be a finite group. Then $\mathrm{geom}(G,G)
\leq \mathrm{ccf}(G) \cdot {(\mathrm{ncf}(G))}^{1/p}$ where $p$ is
the smallest prime factor of the order of $G/F(G)$.
\end{thm}

By taking the reciprocals of both sides of the inequality of
Theorem \ref{t3} and multiplying by $|G|$, we obtain the following
result.

\begin{cor}
\label{c2} Let $G$ be a finite group. Then $\mathrm{ncf}(G) \leq
{\Big( \prod_{g \in G} |\mathrm{cl}_{G}(g)| \Big)}^{p/((p-1)|G|)}$
where $p$ is the smallest prime factor of the order of $G/F(G)$.
\end{cor}

A group is said to be a BFC group if its conjugacy classes are
finite and of bounded size. A group $G$ is called an $n$-BFC group
if it is a BFC group and the least upper bound for the sizes of
the conjugacy classes of $G$ is $n$. One of B. H. Neumann's
discoveries was that in a BFC group the commutator subgroup $G'$
is finite \cite{N1}. One of the purposes of this paper is to give
an upper bound for $|G'|$ in terms of $n$ for an $n$-BFC group
$G$.   Note that $C_G(G')$ is a finite index nilpotent subgroup.
Thus, $F(G)$ is well defined for BFC groups.

If $G$ is a BFC group, then there is a finitely generated
subgroup $H$ with $H'=G'$ and $G=HC_G(G')=HF(G)$.  Then
$H$ has a finite index central torsionfree subgroup $N$.
Set $J=H/N$.  So $J' $ and $G'$  are $G$-isomorphic.   In particular,
 $\mathrm{ncf}(J) = \mathrm{ncf}(G)$.  Clearly,  $G/F(G) \cong J/F(J)$.
Thus, for the next result, it suffices to consider finite
groups. Our first main theorem on BFC groups follows from
Corollary \ref{c2} (by noticing that $|\mathrm{cl}_{G}(1)| = 1$
and that in that result, we may always assume the action is
faithful).

\begin{thm}
\label{t4} Let $G$ be an $n$-BFC group with $n > 1$. Then
$\mathrm{ncf}(G) < n^{p/(p-1)}  \le n^{2}$, where $p$ is the
smallest prime factor of the order of $G/F(G)$.
\end{thm}

Theorem \ref{t4} solves \cite[Conjecture A]{nv}.

Not long after B. H. Neumann's proof that the commutator subgroup
$G'$ of a BFC group is finite, Wiegold \cite{W} produced a bound
for $|G'|$ for an $n$-BFC group $G$ in terms of $n$ and
conjectured that $|G'| \leq n^{(1/2)(1+ \log n)}$ where the
logarithm is to base $2$. Later Macdonald \cite{M} showed that
$|G'| \leq n^{6n{(\log n)}^{3}}$ and Vaughan-Lee \cite{V} proved
Wiegold's conjecture for nilpotent groups. For solvable groups the
best bound to date is $|G'| \leq n^{(1/2)(5 + \log n)}$ obtained
by Neumann and Vaughan-Lee \cite{nv}. In the same paper they
showed that for an arbitrary $n$-BFC group $G$ we have $|G'| \leq
n^{(1/2)(3 + 5 \log n)}$. Using the Classification of Finite
Simple Groups (CFSG) Cartwright \cite{C} improved this bound to
$|G'| \leq n^{(1/2)(41 + \log n)}$ which was later further
sharpened by Segal and Shalev \cite{SS} who obtained $|G'| \leq
n^{(1/2)(13 + \log n)}$. Applying Theorem \ref{t4} at the bottom
of \cite[Page 511]{SS} we arrive at a further improvement of the
general bound on the order of the derived subgroup of an $n$-BFC
group.

\begin{thm}
\label{t5} Let $G$ be an $n$-BFC group with $n > 1$. Then $|G'| <
n^{(1/2)(7+ \log n)}$.
\end{thm}

A word $\omega$ is an element of a free group of finite rank. If
the expression for $\omega$ involves $k$ different indeterminates,
then for every group $G$, we obtain a function from $G^{k}$ to $G$
by substituting group elements for the indeterminates. Thus we can
consider the set $G_{\omega}$ of all values taken by this
function. The subgroup generated by $G_{\omega}$ is called the
verbal subgroup of $\omega$ in $G$ and is denoted by $\omega(G)$.
An outer commutator word is a word obtained by nesting commutators
but using always different indeterminates. In \cite{FM}
Fern\'andez-Alcober and Morigi proved that if $\omega$ is an outer
commutator word and $G$ is any group with $|G_{\omega}| = m$ for
some positive integer $m$ then $|\omega(G)| \leq {(m-1)}^{m-1}$.
They suspect that this bound can be improved to a bound close to
one obtainable for the commutator word $\omega = [x_{1},x_{2}]$.
By noticing that every conjugacy class of a group $G$ has size at
most the number of commutators of $G$ we see that Theorem \ref{t5}
yields

\begin{cor}
\label{c6} Let $G$ be a group with $m$ commutators for some
positive integer $m$ at least $2$. Then $|G'| < m^{(1/2)(7+ \log
m)}$.
\end{cor}

Segal and Shalev \cite{SS} showed that if $G$ is an $n$-BFC group
with no non-trivial abelian normal subgroup then $|G| < n^{4}$. We
improve and generalize this result in Theorem \ref{t6}. For a
finite group $X$ let $k(X)$ denote the number of conjugacy classes
of $X$.

\begin{thm}
\label{t6} Let $G$ be an $n$-BFC group with $n > 1$. If the
Fitting subgroup $F(G)$ of $G$ is finite, then $|G| < n^{2}
k(F(G))$. In particular, if $G$ has no non-trivial abelian normal
subgroup then $|G| < n^{2}$.
\end{thm}

Since $F(G)$ has finite index in $G$,
the hypotheses of Theorem \ref{t6} imply that $G$ is finite.
Note that
even more is true than Theorem \ref{t6}; if $G$ is a finite group
then $|G| \leq a^{2} k(F(G))$ where $a = |G|/k(G)$ is the
(arithmetic) average size of a conjugacy class in $G$ (this is
\cite[Theorem 10 (i)]{GR}). If $b$ denotes the maximal size of a
set of pairwise non-commuting elements in $G$ then, by Tur\'an's
theorem \cite{T} applied to the complement of the commuting graph
of $G$, we have $a < b+1$. Thus if $G$ is a finite group with no
non-trivial abelian normal subgroup then $|G| < {(b+1)}^{2}$. This
should be compared with the bound $|G| < c^{{(\log b)}^{3}}$
holding for some universal constant $c$ with $c \geq 2^{20}$ which
implicitly follows from \cite[Lemma 3.3 (ii)]{Pyber} and should
also be compared with the remark in \cite[Page 294]{Pyber} that
for a non-abelian finite simple group $G$ we have $|G| \leq 27
\cdot b^3$.



The final main result concerns $n$-BFC groups with a given number
of generators. Segal and Shalev \cite{SS} proved that in such
groups the order of the commutator subgroup is bounded by a
polynomial function of $n$. In particular they obtained the bound
$|G'| \leq n^{5d+4}$ for an arbitrary $n$-BFC group $G$ that can
be generated by $d$ elements. By applying Theorem \ref{t4} to
\cite[Page 515]{SS} we may improve this result.

\begin{cor}
\label{c3} Let $G$ be an $n$-BFC group that can be generated by
$d$ elements. Then $|G'| \leq n^{3d+2}$.
\end{cor}

Finally, the following immediate consequence of Corollary \ref{c3}
sharpens \cite[Corollary 1.5]{SS}.

\begin{cor}
\label{c4} Let $G$ be a $d$-generator group. Then $$|\{ [x,y] : x,
y \in G \}| \geq {|G'|}^{1/(3d+2)}.$$
\end{cor}

The example $T_{m}(p)$ \cite[Page 213]{nv} shows that Theorem
\ref{t5}, Corollary \ref{c6}, Corollary \ref{c3}, and Corollary
\ref{c4} are close to best possible.

We point out that Theorem \ref{t1} for $p$ odd requires only the
Feit-Thompson Odd Order Theorem \cite{FT}. However, most of the
results in this paper depend on CFSG as do the results in
\cite{SS} and \cite{IKMM} (for groups of even order).

\section{Proof of Theorem \ref{t1}}

Our first lemma sharpens and generalizes \cite[Theorem 6.1]{nv}.

\begin{lem}
\label{lem2} Let $G$ be a finite group, $F$ a field, and $V$ a
finite dimensional $FG$-module. Let $N$ be an elementary abelian
normal subgroup of $G$ such that $C_{V}(N) = 0$. Then
$\mathrm{avgdim}(Ng,V) \leq (1/p) \dim V$ for every $g \in G$
where $p$ is the smallest prime factor of the order of $G$.
\end{lem}

\begin{proof}
Let us consider a counterexample with $|G|$ and $\dim V$ minimal.
It clearly suffices to assume that $G=\langle g, N \rangle$. We
may assume that $V$ is irreducible (since if we have the
inequality on each composition factor of $V$ we have it on $V$).
We may also assume that $V$ is absolutely irreducible. Finally, we
may assume that $N$ acts faithfully on $V$. If $N$ does not act
homogeneously on $V$, then $g$ transitively permutes the
components in an orbit of size $t \geq p$ and so every element in
$Ng$ has a fixed point space of dimension at most $(1/t) \dim V
\leq (1/p) \dim V$. So we may assume that the elementary abelian
group $N$ acts homogeneously on $V$. This means that it acts as
scalars on $V$. Thus $N \leq Z(G)$ and $G/Z(G)$ is cyclic. It
follows that $G$ is abelian and so $\dim V = 1$. At most $1$
element in the coset $Ng$ is the identity and so
$\mathrm{avgdim}(Ng,V) \leq (1/|N|)\dim V \leq (1/p) \dim V$. The
result follows.
\end{proof}

We first need a result about generation of finite groups. This is
an easy consequence of the proof of the main results of
\cite{BGK}.


\begin{thm}
\label{gen} Let $G$ be a finite group with a minimal normal subgroup $N =
L_1 \times \ldots \times L_t$ for some positive integer $t$ with
$L_i \cong L$ for all $i$ with $1 \leq i \leq t$ for a non-abelian
simple group $L$. Assume that $G/N = \langle xN \rangle$ for some
$x \in G$. Then there exists an element $s \in L_{1} \le N$ such
that $|\{g \in Nx : G = \langle g, s \rangle\}| > (1/2)|N|$.
\end{thm}

\begin{proof}
First suppose that $t =1$. This is an immediate consequence of
\cite[Theorem 1.4]{BGK} unless $G$ is one of $Sp(2n,2), n >2$,
$S_{2m+1}$ or $L= \Omega^+(8,2)$ or $A_6$.

If $G=Sp(2n,2), n>2$, then the result follows by \cite[Proposition 5.8]{BGK}. If
$G=S_{2m+1}$, then apply \cite[Proposition 6.8]{BGK}.

Suppose that $L=A_6$. Note that the overgroups of $s$ of order $5$ in $A_6$
are two subgroups isomorphic to $A_5$ (of different conjugacy classes) and the
normalizer of the Sylow $5$-subgroup generated by $s$. The result
follows trivially
from this observation.

Finally consider $L=\Omega^+(8,2)$. We take $s$ to be an element of order $15$.
It follows by the discussion in \cite[Section 4.1]{BGK} that given
$G$, there is an element of order $15$ satisfying the result (although it is
possible that the choice of $s$ depends on which $G$ occurs).

Now assume that $t > 1$. Write $x =(u_1, \ldots, u_t)\sigma$ where
$\sigma$ just cyclically permutes the coordinates of $N$ (sending $L_i$
to $L_{i+1}$ for $i < t$) and $u_i \in
\mathrm{Aut}(L_i)$. By conjugating by an element of the group
$\mathrm{Aut}(L_1) \times \ldots \times \mathrm{Aut}(L_t)$ we may
assume that $u_2 = \ldots = u_t=1$ (we do not need to do this but
it just makes the computations easier).

Let $f:Nx \rightarrow \mathrm{Aut}(L_1)$ be the map sending $wx$ to
the projection of $(wx)^t$ in $\mathrm{Aut}(L_1)$. Write $w=(w_1,
\ldots, w_{t})$
with $w_i \in L_i$. Then $f(wx) = w_tw_{t-1} \ldots w_1u_1$ is in $L_1u_1$.
Moreover, we see that every fiber of $f$ has the same size. By the case $t=1$,
we know that the probability that
$\langle f(wx), s \rangle = \langle L_1, u_1\rangle$ is greater than $1/2$.

We claim that if $L_1 \le \langle (f(wx), s \rangle$, then $G=\langle
wx, s \rangle$. The claim then implies the result. So assume that
$L_1 \le \langle (f(wx), s \rangle$ and set $H=\langle wx, s \rangle$.
Let $M  \le N$ be the normal closure of $s$ in $J: =\langle (wx)^t, s \rangle$.
This projects onto $L_1$ by assumption, but is also contained in $L_1$, whence
$M=L_1$. So $L_1 \le H$. Since any element of $Nx$ acts transitively on the
$L_i$, it follows that $N \le H$ and so $G=H$.
\end{proof}


The next result we need is Scott's Lemma \cite{scott}. See
\cite{neu} for a slightly easier proof of the result which depends
only on the rank plus nullity theorem in linear algebra.

\begin{lem}[Scott's Lemma]
\label{scott} Let $G$ be a subgroup of $\GL(V)$ with $V$ a finite
dimensional vector space. Suppose that $G = \langle g_1, \ldots,
g_r \rangle$ with $g_1 \cdots g_r=1$. Then
$$
\sum_{i=1}^r \dim [g_i, V] \ge  \dim V + \dim [G,V] - \dim C_V(G).
$$
\end{lem}

We will apply this in the case $r=3$. Noting that $\dim V = \dim
[x,V] + \dim C_V(x)$ for any $x$, we can restate this as:

$$
\sum_{i=1}^3  \dim C_V(g_i) \le \dim V + \dim C_V(G) + \dim
V/[G,V].
$$



\begin{thm}
\label{main2} Let $G$ be a finite group. Assume that $G$ has a
normal subgroup $E$ that is a central product of quasisimple
groups. Let $V$ be a finite dimensional $FG$-module for some field
$F$ such that $E$ has no trivial composition factor on $V$. If $g
\in G$, then $\mathrm{avgdim}(gE,V) \le (1/2) \dim V$.
\end{thm}

\begin{proof}
Let us consider a counterexample with $|G|$ and $\dim V$ minimal.
There is no loss of generality in assuming that $F$ is
algebraically closed, $G=\langle E, g \rangle$, and then assuming
that $V$ is an irreducible (hence absolutely irreducible) and
faithful $FG$-module. If $Z(E) \ne 1$, the result follows by Lemma
\ref{lem2} (by taking $N = Z(E)$ and noting that $Z(E)$ is
completely reducible on $V$ with $C_{V}(Z(E)) = 0$ (since $V$ is a
faithful $FG$-module)). So we may assume that $E$ is a direct
product of non-abelian simple groups. If $V$ is not a homogeneous
$FE$-module, then $g$ transitively permutes the homogeneous
components and so any element in $gE$ has fixed point space of
dimension at most $(1/2) \dim V$. So we may assume that $V$ is a
homogeneous $FE$-module. Thus $E=L_1 \times \ldots \times L_m$
with the $L_i$'s non-abelian simple groups. So $V$ is a direct sum
of say $t$ copies of $V_1 \otimes \ldots \otimes V_m$ where $V_i$
is an irreducible nontrivial $FL_i$-module. (Since $G/E$ is cyclic
and $V$ is irreducible, it follows that $t=1$ (by Clifford theory)
but we will not use this fact.) We may replace $E$ by a minimal
normal subgroup of $G$ contained in $E$ (the hypothesis on the
minimal normal subgroup will hold by Clifford's theorem) and so
assume that $g$ transitively permutes the isomorphic subgroups
$L_{1}, \ldots , L_{m}$.




Let $s \in L_1 \le E$ be chosen so that
$Y:=\{ y \in gE$  :  $\langle y, s \rangle = G\}$ has size larger
than $(1/2)|E|$. Such an element exists by Theorem \ref{gen}.
Set $c=\dim C_V(s)$. If $y \in Y$ then, by Lemma \ref{scott}
(applied to the triple $(y, s, (ys)^{-1})$), we have
 $$
 c + \dim C_V(y) + \dim C_V(ys) \le \dim V.
 $$
 For any $y \in Y':= gE \setminus Y$,  we at least have
 $$
 \dim C_V(y) + \dim C_V(ys) \le \dim V + c.
 $$
 Thus, $$2 \sum_{y \in gE} \dim C_{V}(y) = \sum_{y \in gE} \Big( \dim C_V(y) + \dim
 C_V(ys) \Big)$$
 is at most
 $$|Y| (\dim V -  c) + |Y'|(\dim V + c) < |E| \dim V.$$
 This gives $\mathrm{avgdim}(gE) \le (1/2) \dim V$ as required.
 \end{proof}

We now prove Theorem \ref{t1}. As usual, we may assume that $F$ is
algebraically closed, $V$ is an irreducible $FG$-module, and $N$
acts faithfully on $V$. Let $A$ be a minimal normal subgroup of
$G$ contained in $N$. Since $V$ is a faithful completely reducible
$FN$-module, $A$ has no trivial composition factor on $V$. Now
apply Lemma \ref{lem2} and Theorem \ref{main2} to conclude that
$\mathrm{avgdim}(Ag,V) \leq (1/p) \dim V$ where $p$ is the
smallest prime divisor of $|G|$. Since $Ng$ is the union of cosets
of $A$, the result follows.


\section{Proof of Corollary \ref{c5}}




Let us first prove the first statement of Corollary \ref{c5}. By
making the assumptions of the proof of \cite[Corollary D]{IKMM},
it is sufficient to show that the number of $g \in G$ such that
$\dim C_{V}(g) \leq (1/2) \dim V$ is at least
$$\frac{2|G|}{1 + \log_{p}{|G|}_{p}} \leq \frac{2|G|}{2+ \dim V}.$$
But this is clear by Theorem \ref{t1} noting that $\dim V$ is
even.



Let us prove the second statement of Corollary \ref{c5}. Use the
notations and assumptions of the last part of the proof of
\cite[Corollary D]{IKMM}. Let $H$ be a Hall $p'$-subgroup of $G$.
Since $V$ is a completely reducible $G$-module with $C_{V}(G) =
0$, the vector space $V$ is also a completely reducible $H$-module
with $C_{V}(H) = 0$. Hence applying Corollary \ref{c1} to the
$H$-module $V$ we get that there exists $g \in H$ with $\dim
C_{V}(g) < (1/2) \dim V$. So the last displayed inequality of the
proof of \cite[Corollary D]{IKMM} becomes
$$\frac{{|\mathrm{cl}_{G}(g)|}_{p}}{p} \geq \chi(1)^{1/3}$$ since
$\dim V$ is even. From this we get that $p^{3} \chi(1) \leq
{|\mathrm{cl}_{G}(g)|_{p}}^{3}$.

\section{Proof of Theorem \ref{t2}}


Note that $Y$ centralizes $M$ and so there is no loss in working
in $G/Y$ and assuming that $X=M$ is a minimal normal subgroup of
$G$. We can replace $G$ by $\langle M, g \rangle$ and so assume
that $g$ acts transitively on the direct factors of $M$.

What we need to show is that the arithmetic mean of the positive
integers $|C_M(x)|$ for $x \in gX$ is at most $|M|^{1/2}$. If
there are $t
> 1$ direct factors in $M$, then every element in $gM$ has
centralizer at most $|M|^{1/t} \le |M|^{1/2}$ and the result
follows. So assume that $t=1$ and $M$ is simple.

We compute the arithmetic mean of the positive integers $|C_M(x)|$
for $x \in gM$. All elements in a given $M$-conjugacy class in
$gM$ have the same centralizer size. If $h \in gM$, then the
$M$-conjugacy class of $h$ has $|M:C_M(h)|$ elements. Thus, we see
that the arithmetic mean is precisely the number of conjugacy
classes in $gM$. By \cite[Lemma 2.1]{FG}, this is at most $k(M)$, the number
of conjugacy classes in $M$, and again by \cite[Proposition 5.3]{FG}, this is at
most $|M|^{.41} < |M|^{1/2}$, whence the result.

\section{Proof of Theorem \ref{t3}}

Let us fix a chief series for a finite group $G$. Let
$\mathcal{N}$ be the set of non-central chief factors of this
series. Let $p$ be the smallest prime factor of the order of $G/F(G)$.
If $N \in \mathcal{N}$ is abelian then, by Theorem \ref{t1} (noting that
$F(G)$ acts trivially on $N$), we
have $\mathrm{geom}(G,N) \leq {|N|}^{1/p}$. If $N \in \mathcal{N}$
is non-abelian then, by Theorem \ref{t2} and the Feit-Thompson Odd
Order Theorem \cite{FT}, we again have $\mathrm{geom}(G,N) \leq
{|N|}^{1/p}$. Notice also that for any $g \in G$ we have the
inequality $|C_{G}(g)| \leq \mathrm{ccf}(G) \prod_{N \in
\mathcal{N}} |C_{N}(g)|$. From these observations Theorem \ref{t3}
already follows since
$$\mathrm{geom}(G,G) = {\Big( \prod_{g \in G} |C_{G}(g)| \Big)}^{1/|G|}
\leq \mathrm{ccf}(G) {\Big( \prod_{g \in G} \prod_{N \in
\mathcal{N}} |C_{N}(g)| \Big)}^{1/|G|} = $$
$$= \mathrm{ccf}(G) {\Big( \prod_{N \in \mathcal{N}} \prod_{g \in G}
|C_{N}(g)| \Big)}^{1/|G|} = \mathrm{ccf}(G) \Big( \prod_{N \in
\mathcal{N}} \mathrm{geom}(G,N) \Big) \leq$$
$$\leq \mathrm{ccf}(G) \Big( \prod_{N \in \mathcal{N}} {|N|}^{1/p} \Big)
= \mathrm{ccf}(G) \cdot {(\mathrm{ncf}(G))}^{1/p}.$$


\begin{thebibliography}{30}



\bibitem{BGK} T. Breuer; R. M. Guralnick; W. M. Kantor.
Probabilistic generation of finite simple groups, II. \emph{J.
Algebra} Vol. 320. \textbf{2}, (2008), 443-494.

\bibitem{C} M. Cartwright. The order of the derived group of a BFC
group. \emph{J. London Math. Soc. (2)} \textbf{30} (1984),
227-243.

\bibitem{FT} W. Feit; J. G. Thompson. Solvability of groups of
odd order. \emph{Pacific J. Math.} \textbf{13} (1963), 775-1029.

\bibitem{FM} G. A. Fern\'andez-Alcober; M. Morigi. Outer
commutator words are uniformly concise, preprint.

\bibitem{FG} J. Fulman; R. M. Guralnick. Bounds on the number
and sizes of conjugacy classes in finite Chevalley groups with
applications to derangements, preprint.


\bibitem{gm} R. M. Guralnick; G. Malle. Uniform triples and
fixed point spaces, preprint.

\bibitem{GR} R. M. Guralnick; G. R. Robinson. On the commuting
probability in finite groups. \emph{J. Algebra} Vol. 300, No. 2,
(2006), 509-528.

\bibitem{IKMM} I. M. Isaacs; T. M. Keller; U. Meierfrankenfeld;
A. Moret\'o. Fixed point spaces, primitive character degrees and
conjugacy class sizes. \emph{Proc. Am. Math. Soc.} \textbf{134}
11, (2006), 3123-3130.

\bibitem{Kourovka} The Kourovka Notebook. Unsolved problems in
group theory. Sixteenth augmented edition, 2006. Edited by V. D.
Mazurov and E. I. Khukhro.


\bibitem{M} I. D. Macdonald. Some explicit bounds in groups with
finite derived groups. \emph{Proc. London Math. Soc.} (3)
\textbf{11} (1961), 23-56.

\bibitem{N1} B. H. Neumann. Groups covered by permutable subsets.
\emph{J. London Math. Soc.} \textbf{29} (1954), 236-248.


\bibitem{N} Peter M. Neumann. A study of some finite permutation
groups. DPhil thesis, Oxford, 1966.

\bibitem{nv} Peter M. Neumann; M. R. Vaughan-Lee. An
essay on BFC groups. \emph{Proc. London Math. Soc.} (3)
\textbf{35} (1977), 213-237.

\bibitem{neu} M. Neubauer. On monodromy groups of fixed genus. \emph{J.
Algebra} \textbf{153} (1992), 215-261.

\bibitem{Pyber} L. Pyber. The number of pairwise non-commuting elements and the
index of the centre in a finite group. \emph{J. London Math. Soc.
(2)} \textbf{35}, (1987), no. 2, 287-295.

\bibitem{scott} L. Scott. Matrices and cohomology. \emph{Ann. of
Math.} \textbf{105} (1977), 473-492.

\bibitem{SS} D. Segal; A. Shalev. On groups with bounded
conjugacy classes. \emph{Quart. J. Math. Oxford} \textbf{50}
(1999), 505-516.

\bibitem{T} P. Tur\'an. An extremal problem in graph theory.
\emph{Mat. Fiz. Lapok} \textbf{48}, (1941), 436-452.


\bibitem{V} M. R. Vaughan-Lee. Breadth and commutator subgroups of
$p$-groups. \emph{J. Algebra} \textbf{32} (1974) 278-285.

\bibitem{W} J. Wiegold. Groups with boundedly finite classes of
conjugate elements. \emph{Proc. Roy. Soc. London Ser. A}
\textbf{238} (1956), 389-401.

\end{thebibliography}
\end{document}